\documentclass[12pt]{amsart}
\usepackage{amssymb}
\usepackage{amsmath}
\usepackage{amsfonts}
\usepackage{color}
\usepackage{xcolor}

\newtheorem{thm}{Theorem}[section]

\newtheorem{lem}[thm]{Lemma}
\newtheorem{prop}[thm]{Proposition}

\newtheorem{ex}[thm]{Example}
\newtheorem{rem}[thm]{Remark}

\makeatletter
 
 \@addtoreset{equation}{section}
\makeatother
\newcommand{\R}{\mathbb{R}}
\newcommand{\C}{\mathbb{C}}
\newcommand{\Z}{\mathbb{Z}}

\newcommand{\supp}{\operatorname{supp}}
\newcommand{\card}{\operatorname{card}}

\def\spmapright#1{\smash{%
 \mathop{\hbox to 1.3cm{\rightarrowfill}}
  \limits^{#1}}}

\begin{document}

\title[The Square Root Problem and Subnormal Aluthge Transforms]{\bf The Square Root Problem and \\ Subnormal Aluthge Transforms of \\ Recursively Generated Weighted Shifts}
\author[R.E. Curto, H. El Azhar,  Y. Omari and E.H. Zerouali]{R.E. Curto, H. El Azhar, Y. Omari and E.H. Zerouali}
\address{R.E. Curto, Department of Mathematics, The University of Iowa, Iowa City, Iowa, U.S.A.}
\email{raul-curto@uiowa.edu}
\address{H. El Azhar, Faculty of
sciences, Chouaib Doukkali University, El Jadida, Morocco.}
\email{el-azhar.h@ucd.ac.ma}

\address{Y. Omari,  Faculty of Sciences, Ibn Tofail University, Kenitra, Morocco.}
\email{omariysf@gmail.com}

\address{ E.H. Zerouali, Permanent address.  Laboratory of Mathematical Analysis and Applications, Faculty of Sciences, Mohammed V University in Rabat, Rabat, Morocco.}
\address{Current address: Department of Mathematics, The University of Iowa, Iowa City, Iowa, U.S.A.}
\email{elhassan.zerouali@fsr.um5.ac.ma \& ezerouali@uiowa.edu}
\begin{abstract} 
For recursively generated shifts, we provide definitive answers to two outstanding problems in the theory of unilateral weighted shifts: the Subnormality Problem ({\bf SP}) (related to the Aluthge transform) and the Square Root Problem ({\bf SRP}) (which deals with Berger measures of subnormal shifts). \ We use the Mellin Transform and the theory of exponential polynomials to establish that ({\bf SP}) and ({\bf SRP}) are equivalent if and only if a natural functional equation holds for the canonically associated Mellin transform. \ For $p$--atomic measures with $p \le 6$, our main result provides a new and simple proof of the above-mentioned equivalence. \ Subsequently, we obtain an example of a $7$--atomic measure for which the equivalence fails. \ This provides a negative answer to a problem posed by G.R. Exner in 2009, and to a recent conjecture formulated by R.E. Curto et al in 2019.
\end{abstract}

\subjclass {Primary  44A60; 47B37.  Secondary 30C15, 40A99.}
\keywords{Subnormal weighted shift, Aluthge transform, square root problem, Mellin transform, finitely atomic measure.}  
\maketitle

\section{Introduction}
Let $\mathcal{H}$ be an infinite dimensional  Hilbert space and let $\mathcal{L(H)}$ be the space of all bounded linear operators on $\mathcal{H}$. \ An operator  $T \in \mathcal{L(H)}$ is normal if $TT^*=T^*T$, and subnormal if it is the restriction of a normal operator to an invariant subspace. \ (Here $T^*$ stands for the usual adjoint operator of $T$.) \ The polar decomposition of $T$ is given by the unique representation $T= V|T|$, where $|T|:= (T^*T)^{\frac{1}{2}}$ and $V$ is a partial isometry satisfying $ker \, V = ker \, T$. \ The Aluthge transform is then given by the expression 
$$\tilde T:= |T|^{\frac{1}{2}}V|T |^{\frac{1}{2}}.$$

The Aluthge  transform was introduced in  \cite{al}, in
order to extend several inequalities valid for hyponormal operators, and has received ample attention in the last decades.

We consider below the Hilbert space $\mathcal{H}=l^2(\mathbb{Z}_+)$, endowed with the canonical orthonormal basis $\{e_n\}_{n\in\mathbb{Z}_+}$. \ The unilateral (forward) shift operator  $W_{\alpha}$  is defined on the canonical basis by $W_\alpha e_n:=\alpha_n e_{n+1}$, where  $\alpha=(\alpha_n)_{n\geq 0}$ is a given  sequence of positive real numbers (called {\it weights}). \ It is well known that $W_{\alpha}$ is bounded if and only if the sequence of weights is bounded, and  $\displaystyle\|W_{\alpha}\|=\sup_{n\geq0}\alpha_n<+\infty$. \ Clearly, $W_{\alpha}$ is never normal. \ 

\medskip
We associate with $W_\alpha$ the sequence defined by 
$$
\gamma_0:=1  \mbox{  and } \gamma_k\equiv\gamma_k(\alpha):=\alpha_0^2\alpha_1^2\cdots\alpha_{k-1}^2 \mbox{ for } k\geq 1.
$$ 
We will say that a sequence $\gamma=(\gamma_n)_{n\geq0}$ is a moment sequence on $K\subseteq {\mathbb R}$, or that it admits a representing measure $\mu$ supported in $K$,  if 
\begin{equation}\label{cmp}
	\gamma_n = \gamma_n(\mu):= \int_Kt^nd\mu(t)\: \: \mbox{ for every } \: n\ge 0 \:\: \: \mbox{and} \: \: \mathrm{supp}(\mu) \subseteq K.
\end{equation}

 The Berger-Gellar-Wallen Theorem states that  $W_{\alpha}$ is subnormal if and only if there exists a positive Borel measure $\mu$ (called {\it a Berger measure}), representing for $\gamma$ and such that $\mathrm{supp}(\mu)\subseteq[0,\|W_{\alpha}\|^2]$ \cite[III.8.16]{Con}.  In the sequel, when such $\mu$ exists, we will also write $W_{\alpha}=W_{\mu}$, and identify the weighted shift and its Berger measure.

In the literature, it is common to refer to $\gamma$ as the sequence of moments arising from the weight sequence $\alpha$. \ Consequently, the Berger-Gellar-Wallen characterization is usually described as "  $W_\alpha$ is subnormal if and only if the sequence $\gamma$ of moments of $\alpha$ corresponds to the sequence of moments of a positive Borel measure $\mu$." \ To avoid any possible confusion, in this paper we will reserve the phrase "moment sequence" for the sequence of moments of the measure $\mu$. \ With the exception of the discussion in Section 6, throughout the rest of the paper our basic weighted shift $W_\alpha$ will be subnormal, and we will seek necessary and sufficient conditions for the subnormality of the square root shift $W_{\sqrt{\alpha}}$ and of the Aluthge transform $\widetilde{W_\alpha}$. \ 

In the case where $\mu$ is finitely atomic (that is, $\textrm{supp}(\mu)$ is a finite set), there exists a nonzero polynomial $P$ such that $P(\mu)=0$. \ In particular, the sequence $(\gamma_n)_{n\geq0}$ satisfies a recursive relation, and the weighted shift $W_\mu$ is said to be {\it recursively generated}. \ Conversely, if a subnormal weighted shift is recursively generated, then its Berger measure is finitely atomic \cite[Remark 3.10(i)]{CF}.

The Aluthge transform $\widetilde{ W_\alpha}$ of a weighted shift  $W_\alpha$  is also a weighted shift, associated with the sequence ${\tilde  \alpha}_n=\sqrt{\alpha_n\alpha_{n+1}}, \, n\ge 0$. \ Indeed, it is easy to check that  $|W_\alpha| e_n=\alpha_ne_n$ and that $Ve_n=e_{n+1}$. \ It   follows that
$$\widetilde{ W_\alpha} e_n= \sqrt{\alpha_n\alpha_{n+1}}e_{n+1}= W_{\tilde \alpha}e_n.$$
Notice also that ${\tilde \gamma}^2_n= \frac{1}{\alpha_0^2}\gamma_n \gamma_{n+1}$. \ The problem of detecting the subnormality of the Aluthge transform of a weighted shift has been considered by several authors in the two last decades; see, for example, \cite{exn}. \ As observed above, and previously in \cite{curt1}, the subnormality of the square root shift implies both the subnormality of $W_\alpha$ and the subnormality of the Aluthge transform $\widetilde{ W_\alpha}$. \ On the other hand, it is possible to find a weight sequence $\alpha$ such that $W_{\sqrt{\alpha}}$ is {\it not subnormal}, while both $W_\alpha$ and $\widetilde{ W_\alpha}$ are subnormal (Example \ref{ex51}). \ This is the first occurrence in the literature of such a shift, and it exemplifies the significance and usefulness of identifying non-positive charges to represent weighted shifts in a manner resembling (\ref{cmp}). 

The next question, which we call the Subnormality Problem {\bf (SP)}, has been considered in recent papers (cf. \cite[Question 4.1]{exn} and \cite{brz, Curto2019,  SH, SH1}). \\

{\bf (SP)} Under what conditions is the subnormality  of a weighted shift preserved under the Aluthge transform?\\

 The general case remains open and only some partial affirmative results have been obtained in \cite{Curto2019, brz, SH, SH1}. \ For the class of moment infinitely divisible ($\mathcal{MID}$) weighted shifts (i.e., subnormal shifts $W_\alpha$ for which $W_\alpha^t$ remains subnormal for all $0<t<1$), it was proved in \cite{curt2} that the Aluthge transform maps $\mathcal{MID}$ bijectively onto $\mathcal{MID}$, and that $W_\alpha$, its Aluthge transform, and the square root shift are simultaneously either $\mathcal{MID}$ or not $\mathcal{MID}$ (cf. \cite[Corollary 4.7 and Theorem 4.10]{curt2}). \ 
 
 {\bf (SP)} can be reformulated in terms of moment sequences as follows: 
 
 Given a moment sequence $(\gamma_n)_n$ on $K:=[0,M]$, under what conditions is $(\sqrt{\gamma_n\gamma_{n+1}})_n$ also a moment sequence?

 A direct application of Schur's product theorem implies that if $(\sqrt{\gamma_n})_n$ is a moment sequence on $K$, then  $(\sqrt{\gamma_n\gamma_{n+1}})_n$ is also a moment sequence on $K$; this gives, in particular, a sufficient condition to solve {\bf (SP)}. \ The question of whether the reverse implication holds is stated in several recent papers. \ The  next conjecture  appears in \cite[Conjecture 4.6]{Curto2019}:

{\bf Conjecture A}. \ Let $W_\mu$ be a recursively generated (i.e., $\card \supp \mu < \infty$) subnormal weighted shift. \ Then the following statements are equivalent: 
\begin{itemize}
    \item[(i)] \ $W_{\sqrt{\alpha}}$ is subnormal;
    \item[(ii)] \ $\widetilde{ W_\mu}$ is subnormal.
\end{itemize}

\vspace{2pt}
Until now, the treatment of {\bf (SP)} and of Conjecture A has focused on the number of atoms in the $\supp \mu$; let $p:=\card \supp \mu$. \ Using algebraic proofs, affirmative answers to Conjecture A have been obtained for $p\le 6$. \ The case $p=2$ was treated in \cite{SH}, and the case $p=3,4$ in \cite{SH1} where the conjecture appears for the first time; the case $p=5$ was answered in \cite{Curto2019}, and finally the case $p=6$ was solved in \cite{EEEEE}.

In this paper, we provide concrete examples of $p$--atomic measures disproving  Conjecture A, with $p\in \{7,8,9\}$. \ We also recover the case $p\le 6$ as a simple consequence of a new {\it purely analytic} approach. \ The main tools are the Mellin transform and the theory of exponential polynomials developed by J. F. Ritt in the 1920-1930's. 

The organization of the paper is as follows. \ In the next section, we state the definition of the convolution of measures and the main conjecture related to the so-called Square Root Problem {\bf (SRP)}: Given a positive measure $\mu$, under what conditions does there exist a positive measure $\nu$ such that $\nu*\nu= \mu$\,? \ In  Section 3, we exhibit the role of the Mellin transform and the main properties of exponential polynomials as central tools for {\bf (SP)} and {\bf (SRP)}. \ In Section 4, we apply the results obtained in Section 3 to the {\bf (SP)} and {\bf (SRP)} in the case of finitely atomic measures. \ Finally, Section 5 and 6 include some illustrative examples.

\section{Square Roots of Measures}

Given two positive finite measures $\nu$ and $\mu$, let $*$ denote the multiplicative convolution, defined as follows: 
$$
[\nu*\mu](E):=\int_{\mathbb{R}^2} \chi_E(xy) d\nu(x)d\mu(y),
$$
where $\chi_E$ denotes the characteristic function of the Borel set $E$.

It is easy to check that, for any $n\in \Z_+$,
\begin{equation}\label{mom}
  \int_{\R} t^n d(\mu*\nu)(t)=\int_{\R^2} (st)^n d\mu(t)d\nu(s)  =\left( \int_\R t^n d\mu(t) \right) \left( \int_\R s^n d\nu(s) \right).
\end{equation}

In particular, we get $\gamma_n(\mu*\mu)=\gamma_n^2(\mu)$, where 
$\gamma_n(\mu):= \int_\R t^nd\mu(t)$
is the moment of $\mu$ of order $n$. \ The square root problem is usually written as follows:\\

\noindent {\bf (SRP)}: Given a positive measure $\mu$, under what conditions does there exist a positive measure $\nu$ such that $\nu*\nu= \mu$\,?

In the case of compactly supported measures, and thanks to the Weierstrass density theorem and equation \eqref{mom}, the {\bf (SRP)} can be stated in the next simple form. \\

Let $(\gamma_n)_n$ be a moment sequence. \ Under what conditions is  $(\sqrt{\gamma_n})_n$ also a moment sequence?\\

The close relationship between {\bf (SRP)} and  {\bf (SP)} has already been observed in the following proposition from \cite{Curto2019}.

\begin{prop}\label{lemmeconvolution}
	Let $W_\mu$ be a subnormal weighted shift with associated Berger measure  $\mu$. \ Then  $\widetilde{ W_\mu}$ is subnormal if and only if there exists a $ \mathbb{R}^+$-supported probability measure $\nu$ such that $\nu*\nu=\mu*t\mu$.  
\end{prop}
It follows from the previous proposition that 
$$ \mu \mbox{ has a square root } \Rightarrow {\widetilde{ W}_\mu} \mbox{ is subnormal} $$.

The question of whether the reverse implication holds is our main motivation. \ We are naturally led to the following recent conjecture from \cite{Curto2019}.  

{\bf Conjecture A} (cf. \cite[Conjecture 4.6]{Curto2019}). \ 
Let $\mu$ be a finitely atomic Berger measure with support in $\R_+$. \ Then the following statements are equivalent:
\begin{itemize}
    \item[(i)] \ $\mu$ has a square root;
    \item[(ii)] \ $\mu*t\mu$ has a square root.
\end{itemize}

We will now use the next technical results from \cite{EEEEE}, associated with the square root problem of measures; this will allow us to simplify our proof.

\begin{lem}\label{121}
    Let $\mu$ be a positive measure such that $0\notin \mathrm{supp}(\mu)$ and let $s>0$. \ Also, let $\mu_s(t):= \mu(st)$  be    the image measures by the mapping $t\to st$, and
    let $\nu:=a\delta_0+\mu$, where $\delta_\alpha$ denotes  the Dirac measure with atom $\alpha$. \ Then the following statements are equivalent:   \begin{enumerate}
        \item \ $\mu$ \; (resp. $\mu*t\mu$) admits a square root;
        \item \ $\nu$ \; (resp. $\nu*t\nu$) admits a square root;
        \item \ $\mu_s$ \; (resp. $\mu_s*t(\mu_s)$) admits a square root.
    \end{enumerate}
\end{lem}
\medskip
Without loss of generality, hereafter we will assume that $
x_0:=\min(\textrm{supp}(\mu))=1$; in particular, the support of $\mu$ will be contained in $[1,+\infty)$.

\bigskip
We conclude this section with a diagram that illustrates how various conditions for measures and charges are related to the subnormality of $W_\alpha$, $W_{\sqrt{\alpha}}$ and $\widetilde{W_\alpha}$.

($\mathcal{MID}$ level):  
$$
W_{\sqrt{\alpha}} \textrm{ is } \mathcal{MID}
\Longleftrightarrow
W_{\alpha} \textrm{ is } \mathcal{MID}
\Longleftrightarrow
W_{\widetilde{\alpha}} \textrm{ is } \mathcal{MID}.
$$

\bigskip
(Subnormal level: $W_\alpha \sim \mu$)

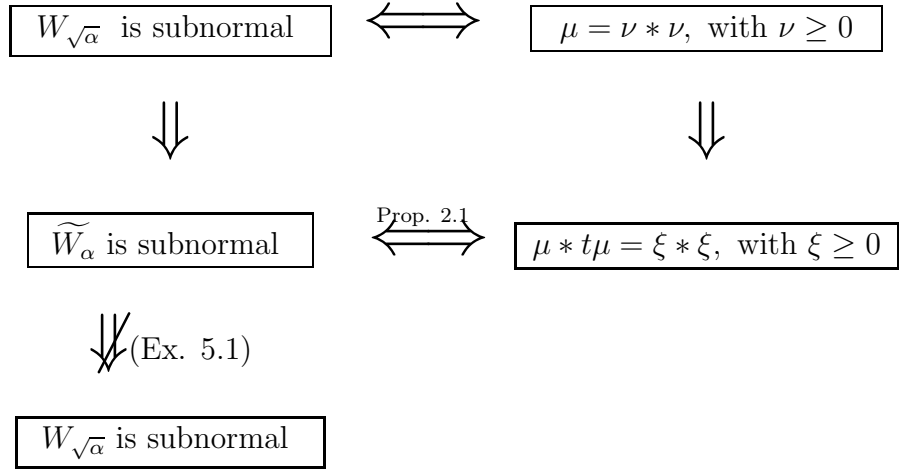
\[\begin{array}{ccc} 
\framebox[1.2\width][c]{$W_{\sqrt{\alpha}}$ \textrm{ is subnormal }} \par & \mbox{\Huge{$\Longleftrightarrow$}}
 & \framebox[1.2\width][c]{$\mu = \nu * \nu, \textrm{ with } \nu \ge 0$} \par 
  \\ & & \\ \mbox{\Huge{$\Downarrow$}} & & \mbox{\Huge{$\Downarrow$}} 
  \\ & & \\ 
  \framebox[1.2\width][c]{$\widetilde{W_\alpha} \textrm{ is subnormal }$} \par & \overset{\textrm{Prop. 2.1} \vspace{4pt}}{\mbox{ \Huge{$\Longleftrightarrow$}}}
 & \framebox[1.1\width][c]{$\mu * t \mu = \xi * \xi, \textrm{ with } \xi \ge 0$} \par \\ & & \\ 
 \mbox{\Huge{$\not\Downarrow$}} (\textrm{Ex. 5.1})  & & \\
 & & \\
 \framebox[1.2\width][c]{$W_{\sqrt{\alpha}} \textrm{ is subnormal }$} \par & &
 \end{array}\]

\bigskip

\begin{rem}
(i) In the above diagram, Example 5.1 is the first known example in the literature. \ \newline
(ii) Since in the above diagram we are assuming that $W_\alpha$ is subnormal, Example 5.1 also exhibits the first example of a subnormal weighted shift with subnormal Aluthge transform for which the square root shift $W_{\sqrt{\alpha}}$ is not subnormal, and therefore not $\mathcal{MID}$. \newline
(iii) For $\mu \ge 0$, one can always find a charge $\rho$ such that $\mu = \rho * \rho$. \newline
(iv) \ Moreover, in Example 6.1 we will prove that the subnormality of $\widetilde{W_\alpha}$ does not necessarily imply the subnormality of $W_\alpha$; indeed, it is possible to find a non-positive charge $\rho$ such that $\rho* t \rho = \nu * \nu$, where $\nu \ge 0$. \newline
(v) In general, it is not true that $\rho * t \rho \ge 0$ implies $\rho \ge 0$; that is, the map $\rho \mapsto \rho * t \rho$ cannot be used as a test for the positivity of $\rho$. 
\end{rem}
\bigskip

\section{The Mellin Transform and Its Relationship to the Aluthge Transform}

Let $\mu$ be a finite positive Radon measure. \ The {\it Mellin transform} $\mathcal{M}_\mu$ is defined as
\begin{equation}\label{Mellindef}
    \mathcal{M}_\mu(z):=\int_{\mathbb{R}_+^*} t^z d\mu(t). 
\end{equation}

We now let 
$$
\mathcal{M}_+(\mathbb{R}_+):= \{ \mathcal{M}_\mu : \mu  \mbox{ is a finite positive Radon measure supported in } \mathbb{R}_+\}.
$$
Using the Perron inversion formula, as in, e.g., \cite[VI Theorem 9b]{Wid}, one can establish that the Mellin transform is one-to-one, and thus it characterizes the measure.
We also have
$$
\mathcal{M}_{\mu*\nu}(z)=\int_{\mathbb{R}_+^2} (uv)^z d\mu(u)d\nu(v)=\int_{\mathbb{R}_+} u^z d\mu(u)\int_{\mathbb{R}_+} v^z d\nu(v)=\mathcal{M}_{\mu}(z)\mathcal{M}_{\nu}(z).
$$
Conjecture A is then equivalent to: \newline \vspace{4pt} {\bf Conjecture B}
\begin{equation} \label{Conjecture B}
\begin{array}{c}
    \mathcal{M}_{\mu*t\mu}(z)=[\mathcal{M}_{\nu}(z)]^2 \\
     \Updownarrow\\
     \textrm{ there exists } \xi \in \mathcal{M}_+(\mathbb{R}_+) \textrm{ with } \mathcal{M}_{\mu}(z)=[\mathcal{M}_{\xi}(z)]^2.  
\end{array}   
\end{equation}

 In addition, we remark that  $$
\mathcal{M}_{t\mu}(z)=\int_{\mathbb{R}_+} t^{z+1} d\mu(t)=\mathcal{M}_{\mu}(z+1),
$$
 for every $z\in\C$ and hence 
$$
\mathcal{M}_{\mu*t\mu}(z)=\mathcal{M}_{\mu}(z)\mathcal{M}_{\mu}(z+1).
$$

Finally, Conjecture A is also equivalent to: \newline
\vspace{4pt} {\bf Conjecture C} \\
\begin{equation} \label{Conjecture C}
\begin{array}{c}
    \mathcal{M}_{\mu}(z)\mathcal{M}_{\mu}(z+1)=[\mathcal{M}_{\nu}(z)]^2 \\
    \Updownarrow\\
    \textrm{ there exists } \xi \in \mathcal{M}_+(\mathbb{R}_+) / \mathcal{M}_{\mu}(z)=[\mathcal{M}_{\xi}(z)]^2. 
\end{array}  
\end{equation}

For $\mu=\sum_{k\ge0}  a_k\delta_{x_k}$ a finite discrete (positive) measure  with compact support in the interval $(0,+\infty)$,  the Mellin transform of $\mu$ is the Dirichlet series
\begin{equation}
    \mathcal{M}_\mu(z)=\sum_{k\ge 0}  a_k{x_k}^z=\sum_{k\ge 0}  a_ke^{z\ln(x_k)}, \label{Mellin}
\end{equation}
which converges uniformly on every compact set of the complex plane. \ Indeed, for every $R>0$ we have
$$
\sum_{k\ge 0} \sup_{z\in D(0,R)}|a_k{x_k}^z|=\sum_{k\ge 0} a_k \sup_{z\in D(0,R)}e^{\Re(z)\ln{x_k}}\le \sum_{k\ge 0} a_k e^{R|\ln{x_k}|}\leq \|\mu\|e^{R\ln{M}},
$$
where $\|\mu\|:=\sum_{k\ge 0} a_k $ stands for the total variation of $\mu$ and $M$ is a positive number such that $|\ln x_k| \leq \ln M$, for every  $k\in\Z_+$. \ In particular $\mathcal{M}_\mu(z)$ is an entire function.

To deal with our main problem, we study assertion \eqref{Conjecture B}. To this end, we need two auxiliary results.
\begin{lem}\label{lem22} Let $\mu=\sum_{k=0}^\infty a_k\delta_{x_k}$ be a positive compactly supported measure in ${\mathbb R}_+$ such that $x_{i_0}=x_{\textrm{min}}:=\inf(\mathrm{supp}(\mu))$ and $x_{i_1}=x_{\textrm{max}}:=\sup(\mathrm{supp}(\mu))$ are isolated in $ \mathrm{supp}(\mu)$. \ Then 
$\mathcal{Z}(\mathcal{M}_{\mu})$, the zero set of $\mathcal{M}_\mu$,  has a bounded real part.
\end{lem} 

\begin{proof} We have 
$$\mathcal{M}_\mu(z)=\sum_{k\ge 0}  a_kx_k^z =a_{i_1}x_{i_1}^z\left[1+ \sum_{k\ne i_1}  \frac{a_k}{a_{i_1}}(\frac{x_k}{x_{i_1}})^z \right], $$ 
and since $$\lim\limits_{\Re(z)\to +\infty}\sum\limits_{k\ne i_1}| \frac{a_k}{a_{i_1}}(\frac{x_k}{x_{i_1}})^z| = \lim\limits_{\Re(z)\to +\infty}\sum\limits_{k\ne i_1}  \frac{a_k}{a_{i_1}}e^{\ln(\frac{x_k}{x_{i_1}})\Re(z)}=0, $$
we deduce that $\mathcal{M}_\mu(z)\ne 0$ for $\Re(z)$ large enough. \ Using $x_{min}$ instead of $x_{max}$, we obtain  similarly that $\mathcal{M}_\mu(z)\ne 0$ for $-\Re(z)$  large enough. \ This completes the proof.
\end{proof}

\begin{prop}\label{lem21} Let $W_\mu$ be a subnormal weighted shift, with $\mu$ a discrete Berger measure as in the previous lemma, and assume that $\widetilde{ W_\mu}$ is also subnormal. \ Let $(z_k,m_k)_k$ the family of zeros, and respective multiplicities, of 
    $\mathcal{M}_{\mu}(z)$. \ Then $m_k$ is even for every $k$. 
\end{prop}

\begin{proof}
    From Equation \eqref{Conjecture B}, we have $\mathcal{M}_{\mu}(z)\mathcal{M}_{\mu}(z+1)=[\mathcal{M}_{\nu}(z)]^2$. \ On the set $\Omega\subset\mathbb{C}$ where $\mathcal{M}_{\mu}(z)$ is holomorphic, we obtain:
 $$    2\frac{\mathcal{M}_{\nu}'(z)}{\mathcal{M}_{\nu}(z)}=\frac{\mathcal{M}_{\mu}'(z)}{\mathcal{M}_{\mu}(z)}+\frac{\mathcal{M}_{\mu}'(z+1)}{\mathcal{M}_{\mu}(z+1)}.
 $$
Using Cauchy's argument principle, we derive that
\begin{equation}\label{zeros}
2 m(z,\mathcal{M}_{\nu}(z))=m(z,\mathcal{M}_{\mu})+m(z+1,\mathcal{M}_{\mu}),
\end{equation}
where $m(z,f)$ is the  multiplicity of the zero $z$ in $f$ (with $m(z,f)=0$ if $z$ is not a zero of $f$).

Seeking a contradiction, assume that the zero set of odd multiplicity  $ \mathcal{Z}_{odd}(\mathcal{M}_\mu)$ is nonempty and let $z\in \mathcal{Z}_{odd}(\mathcal{M}_\mu)$. \ From Equation \eqref{zeros}, we derive that $\{z-1, z+1\}\subset \mathcal{Z}_{odd}(\mathcal{M}_\mu)$ and thus, by induction, $z+\mathbb{Z}\subseteq \mathcal{Z}_{odd}(\mathcal{M}_\mu)$. \ This last statement is false, using Lemma \ref{lem22}. \ This completes the proof.
\end{proof}

We now derive the next preparatory result.

\begin{prop} Let $W_\mu$ be a subnormal weighted shift, with $\mu$ a discrete Berger measure as in the previous lemma, and assume that $\widetilde{ W_\mu}$ is also subnormal. \ Assume also that the assumptions of Lemma \ref{lem22} are satisfied. \ Then 
$$
\mathcal{M}_{\mu}(z)=[H(z)]^2,
\qquad z\in\C,$$
for some entire function  $H$.
\end{prop}

\begin{proof}
    Given a set $I \subseteq \mathbb{Z}_+$, let $ \mathcal{Z}(\mathcal{M}_{\mu}):= \{(z_k,m_k), \: \: k \in I\}$. \  Since the multiplicities of all zeros of the entire function $\mathcal{M}_{\mu}(z)$ are even (using a simple factorization by $(z-z_k)^{m_k}$ in the finite case, or using the Weierstrass factorization theorem in the infinite case), we obtain the desired result.
    \end{proof}
    
To reach our main theorem, we need one more auxiliary result. \ Consistent with the prevailing terminology, we will refer to signed measures as {\it charges}; these are Borel measures that are not necessarily positive. \ Thus, a charge $\xi$ typically admits either an atom with negative density or a Borel set $E$ for which $\xi(E)<0$.

\begin{thm}\label{signedroot} Let $W_\mu$ be a subnormal weighted shift, with $\mu$ a discrete Berger measure, and assume that $\widetilde{ W_\mu}$ is also subnormal. \ Then there exists a finitely atomic charge $\xi$ supported in $[1,+\infty)$ such that 
$$
H(z)=\mathcal{M}_{\xi}(z).
$$
In particular, $\xi*\xi=\mu$.
\end{thm} 

\begin{proof}
Consider $\mu=\sum_{k=0}^p a_k\delta_{x_k}$ such that  $\textrm{supp}(\mu)= \{1=x_0<x_1<\cdots <x_p \}$.

The Mellin transform $\mathcal{M}_\mu$ is an exponential polynomial with nonnegative exponents:
$$
\mathcal{M}_{\mu}(z)=\sum_{k=0}^p a_k{x_k}^z=\sum_{k=0}^p a_k e^{z\ln(x_k)}.
$$
From the previous  discussion,  there exists an entire function $H$ satisfying $H(z)^2= \sum_{k=0}^p a_k{x_k}^z$, for all $z\in\C$.

Next, we use a suitable version of a theorem due to J.F. Ritt. \ As a consequence, we prove that the square root of a positive Borel measure always exists, if we allow charges as solutions.

\begin{thm}\cite[Theorem I]{Ritt}\label{rittt}
    Let $P_k$ be exponential polynomials and  $f$ be an analytic solution, in a circular sector with opening greater than $\pi$, of the equation 
    $$
    f^n+P_{n-1}f^{n-1}+\dots+P_0=0.
    $$
    Then $f$ is also an exponential polynomial, whose exponents are linear combinations of the exponents in the $P_k$'s, and with rational coefficients.
\end{thm}

In particular, from the equation $$
H(z)^2= \sum_{k=0}^p a_k{x_k}^z :=-P_0,
$$
it follows that $H(z)$  is also an exponential polynomial. \ That is, there exist $b_i\in \mathbb{C}$ and $y_i\in \mathbb{R}$, such that $H(z)=\sum_{k=0}^q b_k e^{zy_k}$. \ (Here $\{ y_0<y_1< \cdots < y_q\} $ are  linear combinations of $\ln(x_k)$ with rational coefficients; in particular, all $y_k$'s are real numbers). \  Moreover, using the uniqueness of the representation of exponential polynomials, we get 
$$
\{y_k+y_l, \, 0\le k, l\le q \}= \{\ln(x_0)<\ln(x_1))<\cdots <\cdots<\ln(x_p) \} \subseteq {\mathbb R}_+.
$$
Since   $2y_0= \ln(x_0)= 0$, we obtain  $\{ y_0<y_1< \cdots < y_q\} \subset {\mathbb R}_+.$ 

Finally, $\xi= \sum_{k=0}^q b_k \delta_{e^{y_k}}$ is a charge satisfying $\xi*\xi=\mu$, and such that  $supp(\xi) =\{e^{y_0}, \cdots,e^{y_q} \}\subset [1,+\infty)$.

\end{proof}

\section{Applications to {\bf (SRP)}.}
We begin with the following observation.
\begin{prop}
(i) \ Under the notations above, if $\xi_1$ and $\xi_2$ are  square roots of $\mu$, then $\xi_1 = \pm \xi_2$. \newline
(ii) \ Let $\xi$ be a signed square root of a finitely atomic measure $\mu$. \ Then $\mu$ has a positive square root if and only if the coefficients in $\xi$ have constant sign.
\end{prop}

\begin{proof} \ 1) If  $\xi_1*\xi_1 = \xi_2*\xi_2$, then
$ \mathcal{M}_{\xi_1}(z)^2 = \mathcal{M}_{\xi_2}(z)^2$
for all $z\in\C$. \ Since  $\mathcal{M}_{\xi_1}$ and $ \mathcal{M}_{\xi_2}$ are entire functions, we deduce that  $\mathcal{M}_{\xi_1}(z) = \pm\mathcal{M}_{\xi_2}(z)=\mathcal{M}_{\pm\xi_2}(z)$, and since the Mellin transform is one-to-one, we get  $\xi_1 = \pm \xi_2$.\\
2) From Theorem \ref{signedroot}, $\mu$ always admits a charge $\xi$ as a square root. \ Because of $(1)$,  $\mu$ has, as two square roots, i.e., $\pm\xi$. \ Thus, $\mu$ has a positive square root ($\xi\ge 0$ or $-\xi\ge 0$) if and only if the densities in $\xi$ have constant sign.
\end{proof}

In the sequel, we focus on positive measures $\mu$ such that $\mu*t\mu$ has a positive square root. \ Let us first consider  $\nu$ a signed (i.e., not necessarily positive) square root of $\mu*t\mu$. \ Taking into account the previous proposition, we investigate when the coefficients in $\nu$ have a constant sign. \ Since $\textrm{supp}(\mu)\subset[1,+\infty)$, we get $\textrm{supp}(\mu*t\mu)\subseteq [1,+\infty)$ and then $\textrm{supp}(\nu)\subset[1,+\infty)$.
\ Using the identity
$$
\mathcal{M}_{\nu}^2(z)=\mathcal{M}_{\mu}(z)\mathcal{M}_{\mu}(z+1)=H(z)^2H(z+1)^2 \Rightarrow \mathcal{M}_{\nu}(z)=\pm H(z)H(z+1),
$$
we get
	\begin{equation*}\begin{array}{ll}
	   \mathcal{M}_{\nu}(z)  & =\pm \left(\sum_{k=0}^q b_k e^{zy_k}\right)\left(\sum_{k=0}^q b_ke^{y_k} e^{zy_k}\right) \\
	     & =\pm \sum_{k,l=0}^q b_kb_le^{y_k} e^{z(y_k+y_l)}\\
      & =\sum_k \left(\sum_{(i,j)\in \Gamma(\gamma_k)} b_ib_je^{y_i}\right) e^{z\gamma_k},
	\end{array}
	\end{equation*}
 where $\Gamma(\gamma_k)=\left\{ (i,j)\in(\mathbb{Z}_+)^2\ :\ 0\leq i,j\leq q\ \mbox{and}\ y_i+y_j=ln(\gamma_k)\right\}$.\\
 Now, writing $ \lambda_i = e^{y_i}$, we get \begin{equation}\label{ur}\nu=\sum_k\left(\sum_{(i,j)\in \Gamma(\gamma_k)}b_ib_j\lambda_i\right)\delta_{\gamma_k}.\end{equation}

We will now use the following useful observation.
\begin{rem} In Equation \ref{ur}, the atom  $\gamma_k:=\lambda_i\lambda_j$ is said to be uniquely represented (in symbols, $\gamma_k \in \mathcal{UR}$) if $\card\Gamma(\gamma_k)\le 2$. \ In this case, when $\gamma_k$ is nonnegative, we readily get that $b_i$ and $b_j$ are of the same sign. \ The use of uniquely represented elements in $supp(\mu)$ will be helpful in the sequel.
\end{rem} 
Our strategy now is to consider a charge $\nu$, such that both $\mu=\nu*\nu$ and $\nu*t\nu$ are positive. \ It will follow in particular that $W_\mu$  and $\widetilde{ W_\mu}$ must be subnormal. \ We will then show that if $\nu$ has at most six atoms, it is necessarily positive, and that if $\nu$ has more than six atoms, then it is not necessarily positive. \ This will provide an affirmative answer to Conjecture A for $p\le 6$ and a negative answer for $p\ge 7$.

We begin with the next two auxiliary lemmas.
\begin{lem}\label{card3}
    Let $\nu=\sum_{k=0}^q a_k\delta_{\lambda_k}$ be a charge such that $\nu*\nu=\sum_{k=0}^p b_k \delta_{\gamma_k}$ and $\nu*t\nu$ are positive measures. \ If  $\card\Gamma(\lambda^2_k)\leq 3$ for some  $k$, then  $b_k\ne0$.
\end{lem}
\begin{proof}
 The case $\card\Gamma(\lambda^2_k)=1$ is trivial since it corresponds to a uniquely represented atom. \  Suppose  $\Gamma(\lambda^2_k)=\{(k,k),(i,j),(j,i) \}$, we get $ \lambda_k^2=\lambda_i \lambda_j$ and if the coefficient $b_k=a_k^2+2a_ia_j=0,$  we will get for the coefficient of  $\lambda^2_k$ in $\nu*t\nu$
 $$
 \lambda_ka_k^2+(\lambda_i+\lambda_j)a_ia_j=\sqrt{\lambda_i\lambda_j}a_k^2+(\lambda_i+\lambda_j)a_ia_j<(\lambda_i+\lambda_j)\left(\frac{a_k^2}{2}+a_ia_j\right)=0.$$
 Contradiction.
\end{proof}
\begin{lem}\label{p4-5}
    Let $\nu=\sum_{k=0}^q a_k\delta_{\lambda_k}$ be a charge such that $\nu*\nu=\sum_{k=0}^p b_k \delta_{\gamma_k}$ and $\nu*t\nu$ are positive measures. \ Then \newline
    (i) \ If $q \ge 4$, then  $p\ge 6$. \newline
    (ii) \ If $q \ge 5$, then  $p\ge 7$.
\end{lem}

\begin{proof}
(i) \ Suppose  $q\ge 4$. \ Since  $\card\Gamma(\lambda^2_2)\le 3$ and $\card\Gamma(\lambda_{q-1}^2)\le 3$, we obtain 
$\{\lambda_1^2,\lambda_1\lambda_2,\lambda_2^2, \lambda_{q-1}^2,\lambda_{q-1}\lambda_q,\lambda_q^2\} \subset supp(\mu) $ and hence $p\ge 6$.

(ii) \ From the previous item, $p\ge 6$. \ To show that  $p\ge 7$, it suffices to exhibit a new atom.\\
First, if  $\lambda_2^2$ is $\mathcal{UR}$ or if $\lambda_2^2=\lambda_1\lambda_k$ with some $k>4$, $\lambda_1\lambda_3$  becomes an  $\mathcal{UR}$ and hence provides an additional atom in $\nu*\nu$. \ We write then
  $\lambda_2^2=\lambda_1\lambda_3$ and we show that either $\lambda_3^2$ or $\lambda_2\lambda_3$ is the additional atom or produce a new one. \ To this goal, we suppose that neither $\lambda_3^2$ nor  $\lambda_2\lambda_3$ is $\mathcal{UR}$. \ In this case, necessarily $\lambda_2\lambda_3=\lambda_1\lambda_4$ (otherwise   $\lambda_1\lambda_4$ will be the new atom as an $\mathcal{UR}$.)
We write $\lambda_3^2=\lambda_1\lambda_k$, $\lambda_3^2=\lambda_2\lambda_l$ or  $\lambda_3^2=\lambda_1\lambda_k=\lambda_2\lambda_l$ with zero as  corresponding coefficient. \ Since the two first situations will provide a new atom because of Lemma \ref{card3}, we can assume  that $\lambda_3^2=\lambda_1\lambda_k=\lambda_2\lambda_l$. \ Now, multiplying $\lambda_2\lambda_3=\lambda_1\lambda_4$   with $\lambda_3$ gives $l=4$. 

Now, from the identity $\lambda_1\lambda_k=\lambda_2\lambda_l$, we derive that $k\ge 5$.\\
$1)$ In the case where $k>5$ and  $\lambda_1\lambda_5=\lambda_3\lambda_4$, then by multiplying with $\lambda_3$, we get  
$ \lambda_1\lambda_3\lambda_5=\lambda_3^2\lambda_4=\lambda_1\lambda_k\lambda_4$. \ It follows that  $ \lambda_3\lambda_5=\lambda_k\lambda_4$ for some  $k> 5$, which is impossible. \ Then if $k\ne 5$ $\lambda_1\lambda_5$ will give additional atom as an $\mathcal{UR}$ element. \\
$2)$ $k=5$. \ That is $\lambda_1\lambda_4=\lambda_2\lambda_3$, and
$\lambda_3^2=\lambda_1\lambda_5=\lambda_2\lambda_4$. \ For $r=\frac{\lambda_2}{\lambda_1}$, we get $\lambda_k=\lambda_1r^{k-1}$ for every $k\le 5$. \ Now, to provide the $7^{th}$ atom, it suffices to show that either    $ a_2a_3+a_1a_4\ne 0$ or $a_3^2+2a_1a_5+2a_2a_4\ne 0$. \ Seeking a contradiction, suppose that  $ a_2a_3+a_1a_4=a_3^2+2a_1a_5+2a_2a_4= 0$. \ From the inequality $ (r+r^2)a_2a_3+(1+r^3) a_1a_4\ge 0$, we derive that $a_2a_3<0$. \ Otherwise,
$$ 0\le (r+r^2)a_2a_3+(1+r^3) a_1a_4< (1+r^3)(a_2a_3+ a_1a_4 )= 0$$
It follows also that $a_1a_4>0 $ and then   $a_2a_4>0.$
Now, from 
$a_3^2+2a_1a_5+2a_2a_4=0$, we derive that $a_1a_5<a_1a_5+a_2a_4=-a_3^2<0$ and since 
$(1-r^2)^2>r(1-r)^2$ we obtain the next contradiction
$$\begin{array}{lll} 0 & \le & r^2a_3^2+(1+r^4)a_1a_5+(r+r^3)a_2a_4 \\
&=&-r^2(2a_1a_5+2a_2a_4) +(1+r^4)a_1a_5+(r+r^3)a_2a_4 \\
& = &    (1-r^2)^2a_1a_5+r(1-r)^2a_2a_4 \\
&\le & r(1-r)^2(a_1a_5+a_2a_4) < 0.
\end{array}$$
The proof is complete.
\end{proof}
We now state and prove our main result.
\begin{thm}Let  $q, p$ be integers and $\nu=\sum_{k=0}^q b_k\delta_{\lambda_k}$ be a charge such that $\mu=\nu*\nu$ is a positive $p-$atomic measure with $p\le 6$. \ If $\mu=\nu*t\nu\ge 0$, then the coefficients $ b_k$ have a constant sign. \ In particular, $\mu$ admits a positive square root.
\end{thm}
 The proof below provides a new and simple way to recover a solution to the square root problem in the case $p\le 6$. \\
{\it Proof.} 
\begin{itemize}
    \item \ $p= 2$. \ This is a trivial case, since a $2$--atomic measure $\mu$ has a square root if and only if $\mu= a\delta_0+b\delta_\lambda$, with $a, b, \lambda>0$. \ As a result, there is no $2$--atomic positive measure supported in $[1, +\infty)$ such that $\widetilde{ W_\mu}$ is subnormal.
    \item \ $p= 3$. \ In this case $q=2$, $\nu= b_1 \delta_{\lambda_1}+b_2 \delta_{\lambda_2}$ and $\mu =\nu*\nu =  b_1^2 \delta_{\lambda_1^2}+b_1b_2 \delta_{\lambda_1\lambda_2}+ b_2^2 \delta_{\lambda_2^2}$ with $b_1, b_2$ real numbers. \ Since $b_1b_2$ is uniquely represented, it follows that $b_1$ and $b_2$ have the same sign. \ 
    \item \ $p= 4$. \ A $4$--atomic measure has no square root. \ Indeed, assume $\nu$ exists. \ Then, necessarily   $q\ge 3$. \ Now write 
    $$
    \nu= b_1 \delta_{\lambda_1}+b_2 \delta_{\lambda_1}+ \cdots +  b_{q-1} \delta_{\lambda_{q-1}} +  b_q \delta_{\lambda_q},
    $$
    and therefore, 
    $$
    \nu*\nu = b_1^2 \delta_{\lambda_1^2}+2b_1b_2 \delta_{\lambda_1\lambda_2}+b_2^2\delta_{\lambda_2^2}+\cdots+ b_{q-1}^2 \delta_{\lambda_{q-1}^2}+2b_{q-1}b_{q} \delta_{\lambda_{q-1}\lambda_{q}}+ b_{q}^2 \delta_{\lambda_{q}^2}.
    $$
    It is then clear that $p$ should be at least $5$, a contradiction.
    
    \item \ $p\in \{5, 6\}$. \ From Lemma \ref{p4-5}, we obtain  $q < 5$. \ Thus either  $q =3$ or $q=4$.

-  $q= 3$. \ We put $\nu= b_1 \delta_{\lambda_1}+b_2 \delta_{\lambda_2}+b_3 \delta_{\lambda_3}$ and $\mu=\nu*\nu = b_1^2 \delta_{\lambda_1^2}+2b_1b_2 \delta_{\lambda_1\lambda_2}+b_2^2 \delta_{\lambda_2^2}+2b_1b_3 \delta_{\lambda_1\lambda_3}+2b_2b_3 \delta_{\lambda_2\lambda_3}+ b_3^2 \delta_{\lambda_1^3}\ge 0,$ with $b_1, b_2$ and $b_3$ real numbers. \ Since $ \lambda_1\lambda_2$ and $ \lambda_2\lambda_3$ are uniquely represented, it follow that $b_1b_2>0$ and $b_2b_3>0$. \ This gives as above $b_1, b_2$ and $b_3$ have the same sign.

- For $q = 4$, we write $\nu = b_1\delta_{\lambda_1}+\cdots+b_4\delta_{\lambda_4}$ and 
$$
\begin{array}{ccl}
    \mu = \nu *\nu  & =& b_1^2\delta_{\lambda_1^2}+ b_2^2\delta_{\lambda_2^2}+b_3^2\delta_{\lambda_3^2}+b_4^2\delta_{\lambda_4^2}
+2(b_1b_2\delta_{\lambda_1\lambda_2}+b_1b_3\delta_{\lambda_1\lambda_3} \\
     & & +
b_1b_4\delta_{\lambda_1\lambda_4}+b_2b_3\delta_{\lambda_2\lambda_3} +b_2b_4\delta_{\lambda_2\lambda_4}+
b_3b_4\delta_{\lambda_3\lambda_4})\ge 0.
\end{array}
$$ 
As before, $b_1b_2\ge 0$ and $b_3b_4\ge 0$. \ If, moreover $ \lambda_1\lambda_3\in \mathcal{UR}$ or $ \lambda_2\lambda_4\in \mathcal{UR}$, we get $b_1b_3\ge 0$ or $b_2b_4\ge 0$ and then $b_1,b_2,b_3$ and $b_4$  have a constant sign.\\
If not, $ \lambda_2^2=\lambda_1\lambda_3$,  and  $ \lambda_3^2=\lambda_2\lambda_4$, we get $\frac{\lambda_2}{\lambda_1}= \frac{\lambda_3}{\lambda_2}=\frac{\lambda_4}{\lambda_3} (=r)$  which corresponds to the case when the support is contained in a geometric sequence $\lambda_k=r^k$ with $a=\lambda_1>1$. \  Then 

$
\begin{array}{ll}
    \mu & = b_1^2\delta_{a^2r^2}+2b_1b_2\delta_{a^2r^3}+(2b_1b_3+b_2^2)\delta_{a^2r^4}+
2(b_1b_4+b_2b_3)\delta_{a^2r^5} \\
     & + (2b_2b_4 + b_3^2)\delta_{a^2r^6}+2b_3b_4\delta_{a^2r^7}+b_4^2\delta_{a^2r^8}.
\end{array}
$\\

It follows that $ b_1b_4+b_2b_3\ge 0$ and thus $b_1,b_2,b_3$ and $b_4$ have constant sign. \ 
\end{itemize}


\section{Conjecture A Settled in the Negative}

From the previous section if $\nu$ and $-\nu$  are both $q$ atomic non-positive charges, then $\mu=\nu*\nu $ has no positive square root. \ Also, $\nu*t\nu$  is a square root of $\mu*t\mu$. \ Hence, if   $\nu*t\nu$ is positive, we will have $M_\mu$ and $M_{\tilde{\mu}}$ are subnormal. \ This will provide a counter-example to Conjecture A. \ Since for $p\le 6$, the conjecture is valid, we take $p\ge 7$  (and thence, $q\ge 5$), \ where  $p=\card \supp(\mu)$ and $q=\card \supp(\nu)$.\\

Let $  \lambda \in (1, +\infty)$  and  consider the $5$--atomic charge given by 
$$
\nu = b_1\delta_{\lambda}+b_2\delta_{\lambda^2}+b_3\delta_{\lambda^3}+b_4\delta_{\lambda^4}+b_5\delta_{\lambda^5}.
$$
Assume  that $ \nu*\nu$ and $\nu*t\nu $ are both positive. \ We will have
\begin{align*} \nu*\nu = & \, b_1^2\delta_{\lambda^2}+ 2b_1b_2\delta_{\lambda^3} +(2b_1b_3+b_2^2)\delta_{\lambda^4} +2(b_1b_4+b_2b_3)\delta_{\lambda^5}\\ & +(b_3^2+2(b_1b_5+b_2b_4))\delta_{\lambda^6}    
 +2(b_2b_5+b_3b_4)\delta_{\lambda^7}+ (2b_3b_5+b_4^2)\delta_{\lambda^8}\\ & + 2b_4b_5\delta_{\lambda^9} +b_{5}^2\delta_{\lambda^{10}}, \\
    \nu*t\nu = & \; \lambda b_1^2\delta_{\lambda^2}+ (\lambda+\lambda^2)b_1b_2\delta_{\lambda^3} +(\lambda^2 b_2^2+(\lambda+\lambda^3)b_1b_3)\delta_{\lambda^4}+((\lambda+\lambda^4)b_1b_4\\
    & + (\lambda^2+\lambda^3)b_2b_3)\delta_{\lambda^5} +(\lambda b_3^2+((\lambda+\lambda^2)b_1b_5+(\lambda^2+\lambda^4)b_2b_4))\delta_{\lambda^6}\\ & +((\lambda^2+\lambda^5)b_2b_5 
    +(\lambda^3+\lambda^4)b_3b_4)\delta_{\lambda^7} +((\lambda^3+\lambda^5)b_3b_5+\lambda^4 b_4^2)\delta_{\lambda^8} \\ & + (\lambda^4+\lambda^5)b_4b_5\delta_{\lambda^9} +\lambda b_{5}^2\delta_{\lambda^{10}}.
\end{align*} 

As before $b_1b_2>0$ and $b_4b_5>0$. \ Since $b_1b_4+b_2b_3\ge0 $ and $b_2b_5+b_3b_4\ge 0$ we drive that $b_1,b_2,b_4$ and $b_5$ have constant sign. \ Otherwise $b_2b_5<0$ and $b_1b_4<0$ and both signs of $b_3$ will give a contradiction with $b_1b_4+b_2b_3\ge0 $ and $b_2b_5+b_3b_4\ge 0$.

Without loss of generality, we can assume $b_1,b_2,b_4$ and $b_5$ are nonnegative. \ Denote $p\in\{7,8,9\}$  for the number of atoms in $\nu*\nu$. \ If $b_3>0$, then  $p=9$ and in the case where $b_3<0$    the possible zero coefficients are $$   b_2b_5+b_3b_4,\; \mbox{ and } \;  b_1b_4+b_2b_3.$$ 
 
Clearly  $p=7 \iff b_2b_5+b_3b_4=b_1b_4+b_2b_3=0$,  $p=8 \iff \mbox{ either } b_1b_4+b_2b_3=0 \mbox{ or } b_2b_5+b_3b_4=0 $ and $p=9 $ otherwise. \\
Let us now study those instances when $\mu*t\mu$ is positive. \ Since $ \lambda > 1$, we have  $(\lambda^2+\lambda^5)>(\lambda+\lambda^4) > (\lambda^2+\lambda^3)$, and we drive   that   $$(\lambda^2+\lambda^5)b_2b_5+(\lambda^3+\lambda^4)b_3b_4> (\lambda^3+\lambda^4)(b_2b_5+b_3b_4)\ge 0.$$ and  $$(\lambda+\lambda^4)b_1b_4+(\lambda^2+\lambda^3)b_2b_3>(\lambda^2+\lambda^3)(b_1b_4+b_2b_3)\ge 0.$$ 
 Thus, $\mu*t\mu\ge 0$ if and only if 
$$\lambda b_2^2+(1+\lambda^2)b_1b_3 \ge 0 \mbox{ and }  (1+\lambda^2)b_3b_5+\lambda b_4^2\ge 0, $$
equivalently 
\begin{equation}\label{xx}
    \frac{\lambda}{1+\lambda^2}\ge max(\frac{-b_1b_3}{b_2^2},\frac{-b_3b_5}{b_4^2})
\end{equation} 

\begin{ex} \label{ex51} There exists a $9$--atomic positive measure $\mu$ supported in $\mathbb R_+$ such that: $\mu*t\mu$ has a  positive square root, while $\mu$ has no positive square root.  
\end{ex}
{\it Proof.} Let $x$ be a positive real number and $\lambda > 1$. \ 
Consider the $5-$atomic charge $\xi_x$ given by 
$$\xi_x := \delta_{\lambda}+\delta_{\lambda^{2}}-x\delta_{\lambda^{3}}+\delta_{\lambda^{4}}+\delta_{\lambda^{5}}.$$
The coefficients of $\xi_x$ are: $(b_1,b_2,b_3,b_4,b_5)=(1,1,-x,1,1)$.
The finite atomic measure $\mu_x=\xi_x*\xi_x$ given by
$$\begin{array}{ll}
   \mu_x =   & \delta_{\lambda^{2}} + 2\delta_{\lambda^{3}} + (1-2x)\delta_{\lambda^{4}} + (2-2x)\delta_{\lambda^{5}} + (4+x^2)\delta_{\lambda^{6}}  + (2-2x)\delta_{\lambda^{7}}\\
     &  + (1-2x)\delta_{\lambda^{8}} + 2\delta_{\lambda^{9}}+\delta_{\lambda^{10}}, 
\end{array}
$$
has no positive  square root.

It is also clearly a positive $9-$atomic measure if and only if $0<x < \frac{1}{2}$. \ 
On the other hand, $\mu_x*t\mu_x$ possesses  as square root $\nu_x=\xi_x*t\xi_x$ that is positive (because of (\ref{xx})) for any $\lambda$ satisfying 
$$
x\le  \frac{\lambda}{1+\lambda^2}.
$$ 
One can take, for instance, $\lambda = 2$ and $x=\frac{1}{5}$. \ The measure
$$   \mu= \delta_{4} + 2\delta_{8} + \frac{3}{5}\delta_{16} + \frac{8}{5}\delta_{32} + \frac{101}{25}\delta_{64} + \frac{8}{5}\delta_{128} + \frac{3}{5}\delta_{256} + 2\delta_{512}+\delta_{1024} $$
has no positive square root satisfying $\mu*t\mu \ge 0.$

\begin{ex} \label{ex52} \ For $p=7,8$, there exists a $p$--atomic positive measure supported in $\mathbb R_+$ such that: $\mu*t\mu$ has a positive square root measure, but $\mu$ has no positive square root.  
\end{ex}
Let $$\xi=\delta_{\lambda} + \alpha\delta_{\lambda^2} - \delta_{\lambda^3} + \alpha \delta_{\lambda^4} + \beta \delta_{\lambda^5}$$
where $\alpha, \beta,$ and $  \lambda\ne 1$  are positive numbers. For  $\mu=\xi*\xi, $
we have 
$$\begin{array}{ll}
  \mu =    &  \delta_{\lambda^2} + 2\alpha \delta_{\lambda^3} + (\alpha^2-2)\delta_{\lambda^4} + (1 + 2\alpha^2+2\beta)\delta_{\lambda^6} 
    + (2\alpha\beta-2\alpha )\delta_{\lambda^7} \\
     & + (\alpha^2 -2\beta )\delta_{\lambda^8} + 2\alpha\beta \delta_{\lambda^9} + \beta^2\delta_{\lambda^{10}}
\end{array}  
$$
The measure $\mu$ is positive if and only if: $$ \alpha^2\geq 2\beta \geq 2.$$

On the other hand, the coefficients of $\xi$ are: $(b_1,b_2,b_3,b_4,b_5)=(1,\alpha,-1,\alpha,\beta)$.
Again, because of (\ref{xx}), $\nu=\xi *t\xi\ge 0$, if and only if
 $$\frac{\lambda}{1+\lambda^2}\ge max(\frac{1}{\alpha^2},\frac{\beta}{\alpha^2})=\frac{\beta}{\alpha^2}. $$

Thus, for $\beta=2$, $\lambda = 2$ and $\alpha=3$, the $8$ atomic measure 
$$ \mu = \delta_{4} + 6 \delta_{8} + 7\delta_{16}  + 23\delta_{64} 
    + 6\delta_{128}  + 5\delta_{256} + 12\delta_{512} +2\delta_{1024}
$$
and for  $\beta=1$, $\lambda = 2$ and $\alpha=2$, the $7$ atomic measure 
$$ \mu = \delta_{4} + 4 \delta_{8} + 2\delta_{16} + 13\delta_{64} 
     + 4 \delta_{512} +\delta_{1024}
$$ is positive, without a positive square root, and such that  $\mu*t\mu$  has a positive square root.

\section{An Additional Example} 
We present a concrete example of a non-subnormal weighted shift $W_\alpha$ such that $\widetilde{W_\alpha}$ is subnormal. \ 

\begin{ex} \label{ex61} \ We now exhibit a weighted shift $W_\mu$, where $\mu$ is a non-positive charge and such that $\mu*t\mu$ admits a positive square root. \ Taking into account the computations in the previous section, it suffices to find a non-positive charge $\nu$ such that $\mu=\nu*\nu$ is a charge and $ \sqrt{\mu*t\mu}=\nu*t\nu$ is positive. 

Let $\lambda \in (1, +\infty)$  and  consider a $5$--atomic charge given by 
$$
\nu = b_1\delta_{\lambda}+b_2\delta_{\lambda^2} - \delta_{\lambda^3}+b_4\delta_{\lambda^4}+b_5\delta_{\lambda^5},
$$
where $b_1,b_2,b_4$ and $b_5$ are positive numbers. \
 We  have
\begin{align*} \nu*\nu = &\;  b_1^2\delta_{\lambda^2}+ 2b_1b_2\delta_{\lambda^3} +(b_2^2-2b_1)\delta_{\lambda^4} +2(b_1b_4-b_2)\delta_{\lambda^5}+(1+2(b_1b_5+b_2b_4))\delta_{\lambda^6} \\    
& +2(b_2b_5-b_4)\delta_{\lambda^7}+ (b_4^2-2b_5)\delta_{\lambda^8}+ 2b_4b_5\delta_{\lambda^9} +b_{5}^2\delta_{\lambda^{10}}, \\
    \nu*t\nu = & \; \lambda b_1^2\delta_{\lambda^2}+ (\lambda+\lambda^2)b_1b_2\delta_{\lambda^3} +(\lambda^2 b_2^2-(\lambda+\lambda^3)b_1)\delta_{\lambda^4}\\
    &+((\lambda+\lambda^4)b_1b_4-(\lambda^2+\lambda^3)b_2)\delta_{\lambda^5} +(\lambda +((\lambda+\lambda^2)b_1b_5+(\lambda^2+\lambda^4)b_2b_4))\delta_{\lambda^6}  \\
    &+((\lambda^2+\lambda^5)b_2b_5-(\lambda^3+\lambda^4)b_4)\delta_{\lambda^7} +(\lambda^4 b_4^2-(\lambda^3+\lambda^5)b_5)\delta_{\lambda^8} \\ & + (\lambda^4+\lambda^5)b_4b_5\delta_{\lambda^9} +\lambda b_{5}^2\delta_{\lambda^{10}}.
\end{align*} 
It follows that 
$$\nu*t\nu \ge 0 \iff \left\{\begin{array}{cc}
     \lambda^2 b_2^2-(\lambda+\lambda^3)b_1\ge 0  & (\lambda+\lambda^4)b_1b_4-(\lambda^2+\lambda^3)b_2\ge 0 \\
  \lambda^4 b_4^2-(\lambda^3+\lambda^5)b_5 \ge 0   & (\lambda^2+\lambda^5)b_2b_5-(\lambda^3+\lambda^4)b_4
\end{array}\right.$$
Now, using $\lambda^2 b_2^2-(\lambda+\lambda^3)b_1 \le \lambda^2 (b_2^2-2b_1) $ and $\lambda^4 b_4^2-(\lambda^3+\lambda^5)b_5 \le \lambda^4( b_4^2-2b_5)$, it follows that $b_2^2-2b_1\ge 0 $ and $b_4^2-2b_5\ge 0$. \ In this case,  $$\mu \ge 0 \iff b_1b_4-b_2 \ge 0 \; \mbox{ and } b_2b_5-b_4\ge 0  , $$
and as above, 
$$ 
  \nu*t\nu \ge 0 \iff    \frac{\lambda}{1+\lambda^2}\ge max(\frac{b_1}{b_2^2},\frac{b_5}{b_4^2})
 $$
Taking $b_1=b_5=1$ , $b_2=2$ and $  b_4=3$, we obtain 
$$
\mu= \delta_{\lambda^2}+ 4\delta_{\lambda^3} +2\delta_{\lambda^4} +2\delta_{\lambda^5}+15\delta_{\lambda^6}  -2\delta_{\lambda^7}+ 7\delta_{\lambda^8}+ 6\delta_{\lambda^9} +\delta_{\lambda^{10}},
$$
is a charge for every $\lambda$ and $$
\nu*t\nu \ge 0 \iff \frac{\lambda}{1+\lambda^2} \ge \frac{1}{4} \iff \lambda \in (1, 2+\sqrt{3}].
$$
\end{ex}

\bigskip
{\bf Acknowledgments}. \ The first named author was partially supported by NSF grant DMS-2247167. \\ The last named author was partially supported by the Arab Fund Foundation Fellowship Program.~
The Distinguished Scholar Award - File  1026. \ He also acknowledges the mathematics department of the University of Iowa for its kind hospitality during the preparation of this paper.

%
%


\end{document}